\documentclass[graybox]{svmult}

\usepackage{makeidx}         
\usepackage{graphicx}        
\usepackage{multicol}        
\usepackage[bottom]{footmisc}

\usepackage[english]{babel}
\usepackage{amssymb,amstext,amsmath}

\makeindex             

\begin{document}

\title*{Band depths based on multiple time instances}

\author{Ignacio Cascos\inst{1}\and
Ilya Molchanov\inst{2}}

\institute{Department of Statistics, Universidad Carlos III de Madrid, Av. Universidad 30, E-28911 Legan\'es (Madrid), Spain
\texttt{ignacio.cascos@uc3m.es}
\and Department of Mathematical Statistics and Actuarial Science, University of Berne, Sidlerstr. 5, CH-3012 Berne, Switzerland \texttt{ilya.molchanov@stat.unibe.ch}}

\maketitle

\abstract{Bands of vector-valued functions $f:T\mapsto\mathbb{R}^d$ are
  defined by considering convex hulls generated by their values
  concatenated at $m$ different values of the argument. The obtained
  $m$-bands are families of functions, ranging from the conventional
  band in case the time points are individually considered (for $m=1$)
  to the convex hull in the functional space if the number $m$ of
  simultaneously considered time points becomes large enough to fill
  the whole time domain. These bands give rise to a depth concept that
  is new both for real-valued and vector-valued functions.}

\abstract*{Bands of vector-valued functions $f:T\mapsto\mathbb{R}^d$ are
  defined by considering convex hulls generated by their values
  concatenated at $m$ different values of the argument. The obtained
  $m$-bands are families of functions, ranging from the conventional
  band in case the time points are individually considered (for $m=1$)
  to the convex hull in the functional space if the number $m$ of
  simultaneously considered time points becomes large enough to fill
  the whole time domain. These bands give rise to a depth concept that
  is new both for real-valued and vector-valued functions.}

\section{Introduction}
\label{sec:introduction}

The statistical concept of \emph{depth} is well known for random
vectors in the Euclidean space. It describes the relative position of
$x$ from $\mathbb{R}^d$ with respect to a probability distribution on $\mathbb{R}^d$
or with respect to a sample $x_1,\dots,x_n\in\mathbb{R}^d$ from it. Given a
centrally symmetric distribution (for an appropriate notion of
symmetry), the point of central symmetry is the \emph{deepest} point
(center of the distribution), while the depth of outward points is
low.  The concept of depth has been used in the context of
trimming multivariate data, to derive depth-based estimators
(e.g. depth-weighted $L$-estimators or ranks based on the
center-outward ordering induced by the depth), to assess robustness of
statistical procedures, and for classification purposes, to name a few
areas, see \cite{cas10,liu:par:sin99,zuo:ser00a} for extensive
surveys and further references.

Often, the relative position of a point $x$ with respect to a sample
is defined with respect to the convex hull of the sample or a part of
the sample. For instance, the classical concept of the simplicial
depth appears as the fraction of $(d+1)$-tuples of sampled points
whose convex hull contains $x$, see \cite{liu90}. Its population
version is given by the probability that $x$ is contained in the convex
hull of $(d+1)$ i.i.d. copies of the random vector.

In high-dimensional spaces the curse of dimensionality comes into play
and the convex hull of a finite set of sampled
points forms a rather ``thin'' set and so it is very unlikely to
expect that many points belong to it. Even the convex hull of the
whole sample becomes rather small if the space dimension $d$ is much
larger than the sample size $n$. The situation is even worse for
infinite-dimensional spaces that are typical in functional data
analysis. In view of this, a direct generalisation of the simplicial
depth and convex hull depth concepts leads to the situation where most
points in the space have depth zero, see also \cite{kuel:zin13}, who
discuss problems inherent with the half-space depth in
infinite-dimensional spaces, most importantly zero depth and the lack
of consistency, see also \cite{niet:bat16}.

One possible way to overcome such difficulties is to consider the
depth for the collection of function values at any given time argument
value $t$ and then integrate (maybe weightedly) over the argument
space. This idea goes back to \cite{fraim:mun01} and has been further
studied in \cite{claes:hub:slaet:14,nag17}.

Another approach is based on considering the position of a function
relative to the band generated by functions from the sample. The band
generated by real-valued functions is defined as the interval-valued
function determined by the pointwise minimum and maximum of the
functions from the sample. The corresponding band depth has been
studied in \cite{lop:rom09,lop:rom11}. In the multivariate case the
band becomes a set-valued function that at each point equals the
convex hull of the values of functions from the sample, see
\cite{lop:sun:lin:gen14}.  Another multivariate generalisation of the
band depth in \cite{iev:pag13} is based on taking convex combinations
of band depths associated to each component. Yet another multivariate
functional depth concept was studied in \cite{claes:hub:slaet:14} by
integrating the half-space depth over the time domain, see also
\cite{chak:chaud14}. It is argued in \cite{iev:pag13} that the
multivariate setting makes it possible to incorporate other functional data parameters, such as derivatives, into the sample.
It is also possible to combine a function with its
smoothed version, possibly with different bandwidths.

In this paper we suggest a new concept of multivariate functional
depth based on taking convex hulls of the functions' values at
$m\geq1$ time points combined to build a new higher-dimensional
vector. In a sense, this concept pulls together values of the function
at different points and so naturally incorporates the time dependency
effects, and so better reflects the shape of curves.  Two examples at
which these $m$-band depths are used are presented.

The constructions described in Section~\ref{sec:band-depths} remind
very much the conventional simplicial band depth, where the main point
is to check if a point belongs to the convex hull of a subsample. The
underlying convex hull in the functional space is replaced by the
band, as in \cite{lop:rom09}. It is shown that the introduced band
depth satisfies the main properties described in
\cite{claes:hub:slaet:14,lop:sun:lin:gen14}. The theoretical
computation of the $m$-band depth is usually unfeasible, since it
requires computing the probabilities that a point belongs to a convex
hull of random points. Still, its empirical variant is consistent
and rather easy to compute.

\section{Regions formed by samples in functional spaces}
\label{sec:regi-form-sampl}

\subsubsection*{$m$-bands}

Let $\mathbb{E}$ be a linear space of functions $f:T\mapsto\mathbb{R}^d$ whose
argument $t$ belongs to a rather general topological space $T$. For
example, $\mathbb{E}$ may be the family of continuous functions on an
interval $T$ or a collection of $d$-vectors if $T$ is a finite set.

Consider functions $f_1,\dots,f_j\in\mathbb{E}$. The convex hull
$\operatorname{conv}(f_1,\dots,f_j)$ of these functions is the family of
functions $f\in\mathbb{E}$ that can be represented as
\begin{displaymath}
  f(t)=\sum_{i=1}^j\lambda_i f_i(t)\,,\quad t\in T\,,
\end{displaymath}
for some non-negative constants $\lambda_1,\dots,\lambda_j$ that sum
up to one.

If the coefficients $\lambda_1,\dots,\lambda_j$ are allowed to be
arbitrary functions of $t$, we arrive at the family of functions
$f\in\mathbb{E}$ such that, for all $t\in T$, the value $f(t)$ belongs to the
convex hull of $f_1(t),\dots,f_j(t)$. Following \cite{lop:rom09,lop:sun:lin:gen14} for univariate (resp. multivariate) functions, the
set of such functions is called the \emph{band} generated by
$f_1,\dots,f_j$ and is denoted by $\operatorname{band}(f_1,\dots,f_j)$. It is
obvious that
\begin{displaymath}
  \operatorname{conv}(f_1,\dots,f_j)\subset\operatorname{band}(f_1,\dots,f_j).
\end{displaymath}
If $d=1$ (as in \cite{lop:rom09}), then $\operatorname{band}(f_1,\dots,f_j)$
consists of all functions $f$ such that
\begin{equation}
  \label{eq:b1}
  \min_{i=1,\dots,j} f_i(t)\leq f(t)\leq \max_{i=1,\dots,j}
  f_i(t)\,,\quad t\in T\,.
\end{equation}
In order to obtain a set of functions with interior points, one should
avoid the case when the convex hull of $f_1(t),\dots,f_j(t)$ is of a
lower dimension than $d$ at some $t$. In particular, for this $j$
should be greater than $d$.

We define nested families of functions that lie between the band and
the convex hull generated by the sample.

\begin{definition}
  \label{def:1}
  The $m$-band, $\operatorname{band}_m(f_1,\dots,f_j)$, generated by
  $f_1,\dots,f_j\in\mathbb{E}$ is the family of functions $f\in\mathbb{E}$ such
  that, for all $t_1,\dots,t_m\in T$, the vector
  $(f(t_1),\dots,f(t_m))$ belongs to the convex hull of
  $\{(f_i(t_1),\dots,f_i(t_m)), i=1,\dots,j\}$, i.e.
  \begin{equation}
    \label{eq:mband}
    (f(t_1),\dots,f(t_m))
    =\sum_{i=1}^j \lambda_i (f_i(t_1),\dots,f_i(t_m))
  \end{equation}
  for non-negative real numbers $\lambda_1,\dots,\lambda_j$ that sum
  up to one and may depend on $(t_1,\dots,t_m)$.
\end{definition}

\begin{example}[Special cases]
  \label{ex:special}
  If $T=\{t\}$ is a singleton, the functions become vectors in $\mathbb{R}^d$
  and the $m$-band is their convex hull for all $m\geq1$.

  If $T$ is a finite set of cardinality $k$ and $d=1$, then the
  functions $f_1,\dots,f_j$ of $t\in T$ can be viewed as vectors
  $x_i=(x_{i1},\dots,x_{ik})\in\mathbb{R}^k$, $i=1,\dots,j$. The 1-band is the
  smallest hyperrectangle that contains $x_1,\dots,x_j$, which is
  given by $\times[a_l,b_l]$ for $a_l=\min(x_{il},i=1,\dots,j)$ and
  $b_l=\max(x_{il},i=1,\dots,j)$ for $l=1,\dots,k$. The 2-band is
  obtained as the largest set such that its projections on each 2-dimensional
  coordinate plane equals the projection of the convex hull of
  $x_1,\dots,x_k$. The $k$-band coincides with the convex hull of
  $x_1,\dots,x_j$.
\end{example}

If $m=1$ and $d=1$, then we recover the band introduced in
\cite{lop:rom09} and given by \eqref{eq:b1}, so that
$\operatorname{band}(f_1,\dots,f_j)=\operatorname{band}_1(f_1,\dots,f_j)$.

If $f\in\operatorname{band}_m(f_1,\dots,f_j)$, then each convex combination of the
values for $f_1,\dots,f_j$ and $f$ can be written as a convex
combination of the values of $f_1,\dots,f_j$ and so
\begin{displaymath}
  \operatorname{band}_m(f_1,\dots,f_j)=\operatorname{band}_m(f_1,\dots,f_j,f)\,.
\end{displaymath}

The $m$-band is additive with respect to the Minkowski (elementwise)
addition. In particular,
\begin{equation}
  \label{eq:g-plus}
  \operatorname{band}_m(g+f_1,\dots,g+f_j)=g+\operatorname{band}_m(f_1,\dots,f_j)
\end{equation}
for all $g\in\mathbb{E}$.  The $m$-band is equivariant with respect to linear
transformations, that is,
\begin{equation}
  \label{eq:affine}
  \operatorname{band}_m(Af_1,\dots,Af_j)=\{Af:\,f\in\operatorname{band}_m(f_1,\dots,f_j)\}
\end{equation}
for all $A:T\mapsto \mathbb{R}^{d\times d}$ with $A(t)$ nonsingular for all
$t\in T$. If all functions generating an $m$-band are affected by the
same phase variation, the phase of the $m$-band is affected as shown
below,
\begin{equation}
  \label{eq:phase-variation}
  \operatorname{band}_m(f_1\circ h,\dots,f_j\circ h)
  =\{f\circ h:\,f\in\operatorname{band}_m(f_1,\dots,f_j)\}
\end{equation}
for any bijection $h:T\mapsto T$.  If $d=1$ and $\mathbb{E}$ consists of
continuously differentiable functions on $T=\mathbb{R}$, then
$f\in\operatorname{band}_m(f_1,\dots,f_j)$ yields that $f'$ belongs to
$\operatorname{band}_{m-1}(f'_1,\dots,f'_j)$. This can be extended for higher
derivatives.

It is obvious that $\operatorname{band}_m(f_1,\dots,f_j)$ is a convex subset of
$\mathbb{E}$; since the points $t_1,\dots,t_m$ in Definition~\ref{def:1}
are not necessarily distinct, it decreases if $m$ grows. The following
result shows that the $m$-band turns into the convex hull for large $m$.

\begin{proposition}
  Assume that 
  all functions from $\mathbb{E}$
  are jointly separable, that is there exists a countable set $Q\subset T$
  such that, for all $f\in\mathbb{E}$ and $t\in T$, $f(t)$ is the limit of
  $f(t_n)$ for $t_n\in Q$ and $t_n\to t$.
  Then, for each $f_1,\dots,f_j\in\mathbb{E}$,
  \begin{displaymath}
    \operatorname{band}_m(f_1,\dots,f_j)\downarrow \operatorname{conv}(f_1,\dots,f_j)\qquad
    \text{as}\; m\to\infty.
  \end{displaymath}
\end{proposition}
\begin{proof}
  Consider an increasing family $T_n$ of finite subsets of $T$
  such that $T_n\uparrow Q$ and a certain function $f\in\mathbb{E}$.
  If $m_n$ is the cardinality of $T_n$, and $f$ belongs to the
  $m_n$-band of $f_1,\dots,f_j$, then the values $(f(t),t\in T_n)$
  equal a convex combination of $(f_i(t),t\in T_n)$, $i=1,\dots,j$,
  with coefficients $\lambda_{ni}$. By passing to a subsequence,
  assume that $\lambda_{ni}\to\lambda_i$ as $n\to\infty$ for all
  $i=1,\dots,j$. Using the nesting property of $T_n$, we obtain that
  \begin{displaymath}
    f(t)=\sum \lambda_i f_i(t)\,,\quad t\in T_n\,.
  \end{displaymath}
  Now it suffices to let $n\to\infty$ and appeal to the separability
  of $f$.
\end{proof}

Moreover, under a rather weak assumption, the $m$-band coincides with
the convex hull for sufficiently large $m$.  A set of points in the
$d$-dimensional Euclidean space is said to be in general position if
no $(d-1)$-dimensional hyperplane contains more than $d$ points. In
particular, if the set contains at most $d+1$ points, they will be in
general position if and only if they are all extreme points of their
convex hull, equivalently, any point from their convex hull is
obtained as their unique convex combination.

\begin{proposition}
  \label{prop:m-large-gen}
  If $j\leq d(m-1)+1$ and there exists $t_1,\dots,t_{m-1}\in T$ such
  that the vectors $(f_i(t_1),\dots,f_i(t_{m-1}))\in\mathbb{R}^{d(m-1)}$,
  $i=1,\dots,j$, are in general position, then
  \begin{displaymath}
    \operatorname{band}_m(f_1,\dots,f_j)=\operatorname{conv}(f_1,\dots,f_j)\,.
  \end{displaymath}
\end{proposition}
\begin{proof}
  Let $f\in\operatorname{band}_m(f_1,\dots,f_j)$.  In view of \eqref{eq:mband},
  $(f(t_1),\dots,f(t_{m-1}))$ equals a convex combination of
  $(f_i(t_1),\dots,f_i(t_{m-1}))$, $i=1,\dots,j$, which is unique by
  the general position condition. By considering an arbitrary
  $t_m\in T$, we see that $f(t_m)$ is obtained by the same convex
  combination, so that $f$ is a convex combination of functions
  $f_1,\dots,f_j$.
\end{proof}

In particular, if $d=1$, then the $2$-band of two functions coincides
with their convex hull. It suffices to note that if $f_1$
and $f_2$ are not equal, then $f_1(t_1)$ and $f_2(t_1)$
are different for some $t_1$ and so are in general position. The same holds for any dimension $d\geq 2$.

\begin{example}[Linear and affine functions]
  \label{ex:lin}
  Let $f_1,\dots,f_j$ be constant functions. Then their 1-band is the
  collection of functions lying between the maximum and minimum values
  of $f_1,\dots,f_j$. The 2-band consists of constant functions only
  and coincides with the convex hull. Together with (\ref{eq:g-plus}),
  this implies that the $2$-band generated by functions
  $f_i(t)=a(t)+b_i$, $i=1,\dots,j$, is the set of functions $a(t)+b$
  for $b$ from the convex hull of $b_1,\dots,b_j$.

  If $f_i(t)=a_it+b_i$, $i=1,\dots,j$, are affine functions of $t\in \mathbb{R}$,
  then their 3-band consists of affine functions only and also equals
  the convex hull. Indeed,
  \begin{displaymath}
    (f(t_1),f(t_2),f(t_3))=\sum \lambda_i (a_i(t_1,t_2,t_3)+b_i(1,1,1))
  \end{displaymath}
  yields that
  \begin{displaymath}
    \frac{f(t_3)-f(t_1)}{f(t_2)-f(t_1)}
    =\frac{t_3-t_1}{t_2-t_1}\,.
  \end{displaymath}
  Therefore each $f$ from $\operatorname{band}_3(f_1,\dots,f_j)$ is an affine function.
\end{example}

\begin{example}[Monotone functions]
  Let $d=1$ and let $f_1,\dots,f_j$ be non-decreasing (respectively
  non-increasing) functions. Then their 2-band is a collection of
  non-decreasing (resp. non-increasing) functions.
  If all functions $f_1,\dots,f_j$ are convex (resp. concave),
  then their 3-band is a collection of convex (resp. concave)
  functions.
\end{example}

\begin{remark}
  The definition of the $m$-band can be easily extended for subsets
  $F$ of a general topological linear space $\mathbb{E}$. Consider a certain
  family of continuous linear functionals $u_t$, $t\in T$. An element
  $x\in\mathbb{E}$ is said to belong to the $m$-band of $F$ if for each
  $t_1,\dots,t_m\in T$, the vector $(u_{t_1}(x),\dots,u_{t_m}(x))$
  belongs to the convex hull of $\{(u_{t_1}(y),\dots,u_{t_m}(y)):\;
  y\in F\}$. Then Definition~\ref{def:1} corresponds to the case of
  $\mathbb{E}$ being a functional space and $u_t(f)=f(t)$ for $t\in T$.

  While the conventional closed convex hull arises as the intersection
  of all closed half-spaces that contain a given set, its $m$-band
  variant arises from the intersection of half spaces determined by
  the the chosen functionals $u_t$ for $t\in T$.
\end{remark}

\subsubsection*{Space reduction and time share}

The $m$-band reduces to a $1$-band by defining functions on the
product space $T^m$.

\begin{proposition}
  \label{prop:cp}
  For each $j$, the $m$-band $\operatorname{band}_m(f_1,\dots,f_j)$ coincides with
  $\operatorname{band}(f^{(m)}_1,\dots,f^{(m)}_j)$, where
  $f^{(m)}_i:T^m\mapsto(\mathbb{R}^d)^m$ is defined as
  \begin{displaymath}
    f^{(m)}(t_1,\dots,t_m)=(f(t_1),\dots,f(t_m))\,.
  \end{displaymath}
\end{proposition}
\begin{proof}
  It suffices to note that $f^{(m)}(t_1,\dots,t_m)$ belongs to the
  convex hull of $f^{(m)}_i(t_1,\dots,t_m)$, $i=1,\dots,j$, if and
  only if $(f(t_1),\dots,f(t_m))$ belongs to the convex hull of
  $(f_i(t_1),\dots,f_i(t_m))$, $i=1,\dots,j$.
\end{proof}

In the framework of Proposition~\ref{prop:cp}, it is possible to
introduce further bands (called \emph{space-reduced}) by
restricting the functions $f^{(m)}_i$ to a subset $S$ of $T^m$. For
instance, the 1-band generated by functions $f^{(2)}_1,\dots,f^{(2)}_j$
for the arguments $(t_1,t_2)\in\mathbb{R}^2$ such that $|t_1-t_2|=h$ describes the
joint behaviour of the values of functions separated by the lag
$h$. If $m=1$, then the space reduction is equivalent to restricting
the parameter space, which can be useful, e.g. for discretisation
purposes.

It is possible to quantify the closedness of $f$ to the band by
determining the proportion of the $m$-tuple of time values from $T^m$
when the values of $f$ belong to the band.  Define the $m$-band
time-share as
\begin{multline*}
  \operatorname{TS}_m(f;f_1,\dots,f_j)\\=\{(t_1,\dots,t_m)\in T^m:\;
  f^{(m)}(t_1,\dots,t_m)\in\operatorname{conv}(\{f_i^{(m)}(t_1,\dots,t_m)\}_{i=1}^j)\,.
\end{multline*}
If the functions take values in $\mathbb{R}$, then
$\operatorname{TS}_1(f;f_1,\dots,f_j)$ turns into the modified band depth
defined in \cite[Sec.~5]{lop:rom09}.  If $f$ belongs to the $m$-band
of $f_1,\dots,f_j$, then $\operatorname{TS}_m(f;f_1,\dots,f_j)=T^m$, while if
$f$ belongs to the $1$-band of $f_1,\dots,f_j$, then
$\{(t,\dots,t):\;t\in T\}\subset\operatorname{TS}_m(f;f_1,\dots,f_j)$.
It is also straightforward to incorporate
the space reduction by replacing $T^m$ with a subset $S$.

\section{Simplicial-type band depths}
\label{sec:band-depths}

\subsubsection*{Band depth}

In the following,  we consider the event that a function $f$ belongs to
a band generated by i.i.d. random functions $\xi_1,\dots,\xi_j$ with
the common distribution $P$. The $m$-band depth of the function $f$
with respect to $P$ is defined by
\begin{multline}
\label{eq:bdf}
  \operatorname{bd}_m^{(j)}(f;P)={\mathbf P}\{f\in\operatorname{band}_m(\xi_1,\dots,\xi_j)\}\\
  ={\mathbf P}\{(f(t_1),\ldots,f(t_m))\in \operatorname{conv}(\{(\xi_i(t_1),\dots,\xi_i(t_m))\}_{i=1}^j)\,,\forall t_1,\ldots,t_m\in T\}\,.
\end{multline}
If $m$ increases, then the $m$-band narrows, and so the $m$-band
depth decreases.

We recall that when $d=1$ the $1$-band coincides with the band
introduced in \cite{lop:rom09}. Nevertheless the band depth defined in
\cite{lop:rom09} is the sum of $\operatorname{bd}_m^{(j)}(f;P)$ with $j$ ranging
from $2$ to a fixed value $J$. The same construction can be applied to our
$m$-bands.

The $m$-band depth of $f$ is influenced by the choice of $j$, and it
increases with $j$. Unlike to the finite-dimensional setting,
where $j$ is typically chosen as the dimension of the space plus one
\cite{liu90}, there is no canonical choice of $j$ for the functional
spaces. In order to ensure that the $m$-band generated by
$\xi_1,\dots,\xi_j$ differs from the convex hull, it is essential to
choose $j$ sufficiently large, and in any case at least $d(m-1)+2$, see
Proposition~\ref{prop:m-large-gen}. Furthermore, we must impose stronger conditions on $j$ to avoid the zero-depth problem.

\begin{proposition}
  If $j\leq dm$ and the joint distribution of the marginals of $P$ at some fixed $m$ time points is absolutely continuous, then  $\operatorname{bd}_m^{(j)}(\cdot;P)=0$\,.
\end{proposition}
\begin{proof}
   If $j\leq dm$ and $\{(\xi_i(t_1),\dots,\xi_i(t_m))\}_{i=1}^j$ are independent and absolutely continuous in $\mathbb{R}^{dm}$, the probability that any fixed $x\in\mathbb{R}^{dm}$ lies in their convex hull is zero.
\end{proof}

A theoretical calculation of the $m$-band depth given by
\eqref{eq:bdf} is not feasible in most cases. In applications, it can
be replaced by its empirical variant defined in exactly the same way
as in \cite{lop:rom09} for the 1-band case. Let $f_1,\dots,f_n$ be a
sample from $P$. Fix any $j\in\{dm+1,\dots,n\}$ and define
\begin{displaymath}
  \operatorname{bd}_m^{(j)}(f;f_1,\dots,f_n)
  ={{n}\choose{j}}^{-1}\sum_{1\leq i_1<\cdots<i_j\leq n}
  {\mathbf 1}_{f\in\operatorname{band}_m(f_{i_1},\dots,f_{i_j})}\,,
\end{displaymath}
so that $\operatorname{bd}_m^{(j)}(f;f_1,\dots,f_n)$ is the proportion of $j$-tuples
from $f_1,\dots,f_n$ such that $f$ lies in the $m$-band generated by
the $j$-tuple.  The choice of $j$ affects the results. It is
computationally advantageous to keep $j$ small, while it is also
possible to sum up the depths over a range of the values for $j$, as
 in \cite{lop:rom09}.

\subsubsection*{Time-share depth}

Assume now that $T$ is equipped with a probability measure $\mu$, for
example, the normalised Lebesgue measure in case $T$ is a bounded
subset of the Euclidean space or the normalised counting measure if
$T$ is discrete. Extend $\mu$ to the product measure $\mu^{(m)}$ on
$T^m$. Define the time-share depth by
\begin{displaymath}
  \operatorname{td}_m^{(j)}(f;P) ={\mathbf E}\mu^{(m)}(\operatorname{TS}_m(f;\xi_1,\dots,\xi_j))\,.
\end{displaymath}
If $T$ is a subset of the Euclidean space, Fubini's Theorem yields
that the time-share depth is the average of the probability that
$(f(t_1),\dots,f(t_m))$ lies in the convex hull of $j$ points in
$\mathbb{R}^{dm}$,
\begin{multline}
  \label{eq:fubini}
  \operatorname{td}_m^{(j)}(f;P)=\int{\mathbf P}\{(f(t_1),\dots,f(t_m))\in\operatorname{conv}(\{(\xi_i(t_1),\dots,\xi_i(t_m))\}_{i=1}^j)\}\\
  {\rm d}\mu^{(m)}(t_1,\dots,t_m)\,.
\end{multline}

For any $j\in\{dm+1,\dots,n\}$, the empirical time-share depth is given by
\begin{displaymath}
  \operatorname{td}_m^{(j)}(f;f_1,\dots,f_n)
  ={{n}\choose{j}}^{-1}\sum_{1\leq i_1<\cdots<i_j\leq n}
  \mu^{(m)}(\operatorname{TS}_m(f;f_{i_1},\dots,f_{i_j}))\,,
\end{displaymath}

\begin{example}[Univariate case]
  Assume that $T$ is a singleton. Then necessarily $m=1$, the function
  $f$ is represented by a point $x$ in $\mathbb{R}^d$, and the band depth of
  $x$ for $j=d+1$ coincides with the simplicial depth, see
  \cite{liu90}.
\end{example}

\begin{example}
  \label{ex:a+x}
  Let $\xi(t)=a(t)+X$, $t\in T$, where $X$ is a random variable. Then
  $\operatorname{band}(\xi_1,\dots,\xi_j)$ for i.i.d. $\xi_i(t)=a(t)+X_i$,
  $i=1,\dots,j$, is the set of functions bounded above by $a(t)+\max
  X_i$ and below by $a(t)+\min X_i$. Then
  \begin{displaymath}
    \operatorname{bd}_1^{(j)}(a;P)=1-{\mathbf P}\{X>0\}^j-{\mathbf P}\{X<0\}^j.
  \end{displaymath}
  By Example~\ref{ex:lin},
  $\operatorname{band}_2(\xi_1,\dots,\xi_j)$ consists of functions $a(t)+b$ for the
  constant $b\in[\min X_i,\max X_i]$. Only such functions may
  have a positive $2$-band depth.
\end{example}

\begin{example}
  Let now $\xi(t)=a(t)+X$, where $a:T\to\mathbb{R}^d$ and $X$ is an absolutely
  continuous random vector in $\mathbb{R}^d$ which is angularly symmetric about the origin. Then
  \begin{equation}
    \label{eq:wendel}
    \operatorname{bd}_1^{(j)}(a;P)=1-2^{1-j}\sum_{i=0}^{d-1}{j-1\choose i}
  \end{equation}
  being the probability that the origin belongs to the convex hull of
  $X_1,\dots,X_j$, see \cite{wen62}.
\end{example}

\subsubsection*{Properties of the band depths}

\begin{theorem}
  For any $j\geq dm+1$ we have:
  \begin{enumerate}
  \item \emph{affine invariance.}
    $\operatorname{bd}_m^{(j)}(Af+g;P_{A,g})=\operatorname{bd}_m^{(j)}(f;P)$ and
    $\operatorname{td}_m^{(j)}(Af+g;P_{A,g})=\operatorname{td}_m^{(j)}(f;P)$ for all $g\in\mathbb{E}$ and
    $A:T\mapsto \mathbb{R}^{d\times d}$ with $A(t)$ nonsingular for $t\in T$.
  \item \emph{phase invariance.} $\operatorname{bd}_m^{(j)}(f\circ
    h;P^h)=\operatorname{bd}_m^{(j)}(f;P)$ for any one-to-one transformation
    $h:T\mapsto T$, where $P^h(F)=P(F\circ h^{-1})$ for any measurable
    subset $F$ of $\mathbb{E}$ when $h^{-1}$ is the inverse mapping of $h$.
  \item \emph{vanishing at infinity.}  $\operatorname{bd}_m^{(j)}(f;P)\to0$ if the
    supremum of $\|f\|$ over $T$ converges to infinity, and
    $\operatorname{td}_m^{(j)}(f;P)\to 0$ if the infimum of $\|f\|$ over $T$
    converges to infinity.
  \end{enumerate}
\end{theorem}

The affine invariance of both depths follows from the affine
invariance of the $m$-bands, see (\ref{eq:g-plus}), (\ref{eq:affine}),
while the phase-invariance of the band depth follows from
(\ref{eq:phase-variation}).

In practice, the functions are going to be evaluated over a finite set
of time points, thus $T=\{t_1,\dots,t_k\}$ and probability $P$ is a
distribution on $(\mathbb{R}^d)^k$. Furthermore, the sample of functions
$f_1,\dots,f_n$ to be used to determine an empirical $m$-band depth
should have size at least $n\geq j\geq dm+1$\,.

\begin{theorem}
  If $P$ is absolutely continuous, for any $n\geq j\geq dm+1$ we have:
  \begin{enumerate}
    \setcounter{enumi}{3}
  \item \emph{maximality at the center.} if $P$ is angularly symmetric
    about the point $(f(t_1),\dots,f(t_k))$, function $f$ will be the
    deepest with regard to the time-share depth, and
    $\operatorname{td}^{(j)}_m(f;P)=1-2^{1-j}\sum_{i=0}^{dm-1}{j-1\choose i}$.
  \item \emph{consistency.} band depth
    $\sup_{f\in\mathbb{E}}|\operatorname{bd}_m^{(j)}(f;f_1,\dots,f_n)-\operatorname{bd}_m^{(j)}(f;P)|\rightarrow
    0$ $a.s.$ and time-share depth
    $\sup_{f\in\mathbb{E}}|\operatorname{td}_m^{(j)}(f;f_1,\dots,f_n)-\operatorname{td}_m^{(j)}(f;P)|\rightarrow
    0$ $a.s.$
  \end{enumerate}
\end{theorem}

The properties of the time-share depth rely on formula
(\ref{eq:fubini}) that makes is possible to write it as an average of
the probability that a point lies in the convex hull of independent
copies of a random vector. The maximality at center follows from the
main result in \cite{wen62} which determines the probability inside
the integral in (\ref{eq:fubini}), see (\ref{eq:wendel}), while the
consistency can be proved in a similar way to
\cite[Th.3]{lop:sun:lin:gen14} extending the uniform consistency of
the empirical simplicial depth \cite[Th.1]{due92} to the one of the
probability that a point lies in the convex hull of a fixed number of
independent copies of a random vector. Such an extension, which relies
on probabilities of intersections of open half-spaces, can be adapted
to prove the consistency of the empirical $m$-band depth.

\section{Data examples}
\label{sec:simulation-study}

\subsubsection*{Simulated data}

Fig.~\ref{fig:simulated} shows $17$ curves which are
evaluated at $T=\{1,2,\dots,9\}$. Among the $17$ curves, there is a
clear shape outlier (marked as \texttt{d}) that lies \emph{deep}
within the bunch of curves.  Such an outlier will not be detected by
the outliergram from \cite{arr:rom14} due to its high depth value with
regard to both of the $1$-band depth and half-region depth
(see~\cite{lop:rom11}). Nevertheless, its anomalous shape is
detected by any $m$-band depth with $m\geq 2$.

\begin{figure}[h]
\sidecaption[h]
  \includegraphics[scale=0.55]{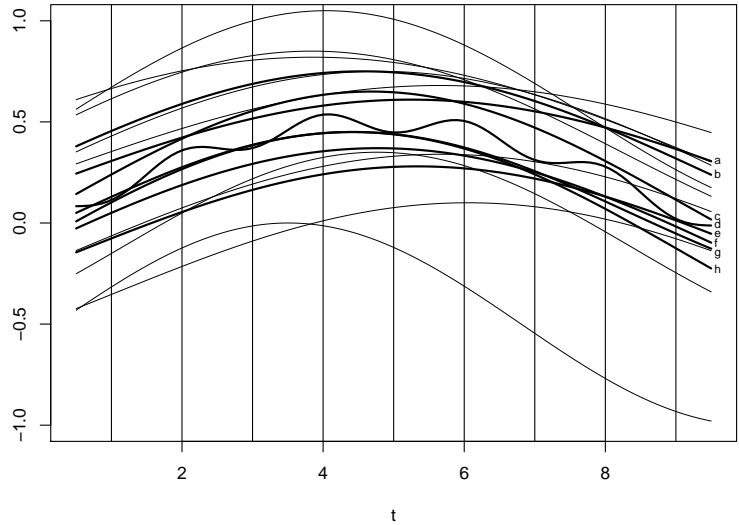}
  \caption{$17$ curves evaluated at $\{1,2,\dots,9\}$. The eight deepest curves are thicker than the others and each of them is assigned a letter from \texttt{a} to \texttt{h}. Five deepest curves for $\operatorname{bd}^{(4)}_1$ (in order): \texttt{d},\texttt{c},\texttt{f},\texttt{g},\texttt{a}, for $\operatorname{td}^{(4)}_1$: \texttt{d},\texttt{c},\texttt{g},\texttt{a},\texttt{f}, for $\operatorname{bd}^{(4)}_2$: \texttt{g},\texttt{b},\texttt{f},\texttt{e},\texttt{a}, and for $\operatorname{td}^{(4)}_2$: \texttt{g},\texttt{f},\texttt{c},\texttt{d},\texttt{h}.}\label{fig:simulated}
\end{figure}

It is remarkable that curve \texttt{d}, which is the deepest curve with respect to
the usual band depth and modified band depth ($\operatorname{bd}^{(4)}_1$ and
$\operatorname{td}^{(4)}_1$) is among the less deep curves for the $2$-band depth
($\operatorname{bd}^{(4)}_2$) and is only the fourth deepest curve for its
time-share depth ($\operatorname{td}^{(4)}_2$). The reason for this last fact is that if we
restrict to either of the sets of time points $\{1,3,5,7,9\}$ or $\{2,4,6,8\}$, curve \texttt{d} is not a shape outlier with respect to them.

\subsubsection*{Real data}

The nominal Gross Domestic Product per capita of the 28 countries of
the European Union (2004--2013) was obtained from the EUROSTAT
web-site and is represented in Fig.~\ref{fig:gdp}. The missing
observation that corresponds to Greece, 2013 was replaced by the value obtained from the FOCUSECONOMICS
web-site.

\begin{figure}[h]
\sidecaption[h]
  \includegraphics[scale=0.5]{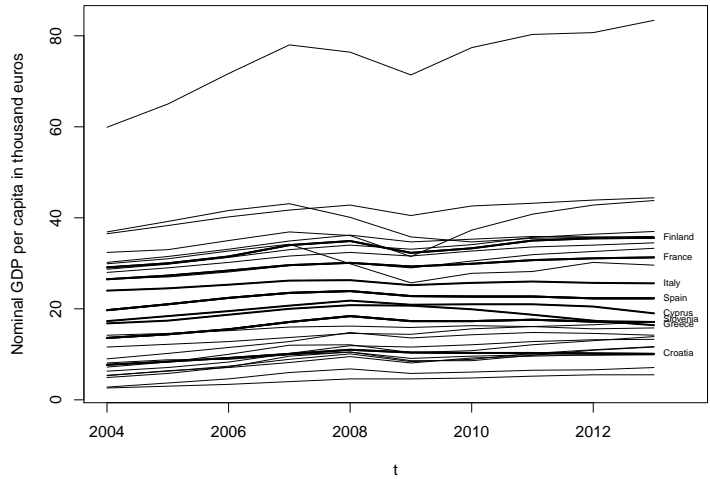}
  \caption{Evolution of the nominal GDP per capita between 2004 and
    2013 at the EU countries. Five deepest curves for $\operatorname{bd}^{(5)}_1$ (in order): Cyprus,
    Spain, Italy, Greece, and Slovenia, for $\operatorname{bd}^{(5)}_2$: Spain,
    Slovenia, France, Croatia, and Finland, and for $\operatorname{bd}^{(5)}_2$
    space-reduced with $S=\{(t_1,t_2):\,|t_1-t_2|=1\}$: Croatia,
    Slovenia, Spain, Finland, and France.}\label{fig:gdp}
\end{figure}

The deepest curve with regard to the band depth
($\operatorname{bd}^{(5)}_1$) is the one of Cyprus. Interestingly, Cyprus suffered the 2012-13 Cypriot
financial crisis at the end of the considered period and its GDP per capita experienced a decay in 2013 in comparison with its 2012 figure much greater than the one of any other of the EU countries. Also the Greek
curve is among the five deepest ones for $\operatorname{bd}^{(5)}_1$ despite being the only country with a constant decrement in the second half of the considered time period. If we consider
$2$-bands, that take into account the shape of the curves, these two
curves are not any more considered representative of the evolution of
the GDP per capita in the EU.

\begin{acknowledgement}

The authors would like to thank Mar\'{\i}a \'Angeles Gil for the
opportunity to contribute in this tribute to Pedro Gil, to whom we do
sincerely appreciate.

Most of this work was carried over while IM was supported by the Chair
of Excellence Programme of the University Carlos III and the Santander
Bank. At that time both authors benefited from discussions with
Professor Juan Romo. IM is grateful to the Department of Statistics of
the University Carlos III in Madrid for the hospitality.
\end{acknowledgement}

\bibliographystyle{spbasic}

\end{document}